\documentclass[twoside,a4paper,reqno,11pt]{amsart}
\usepackage{amsfonts, amsbsy, amsmath, amsthm, amssymb, latexsym, verbatim, enumerate}
\usepackage{mathrsfs}
\usepackage[top=30mm,right=30mm,bottom=30mm,left=30mm]{geometry}

\usepackage{hyperref}
\usepackage{amssymb,amsfonts,amsmath,amsopn,amstext,amscd,latexsym,amsthm}
\usepackage{bbm}
\usepackage[all,cmtip]{xy}

\usepackage{enumerate}
\usepackage[shortlabels]{enumitem}

\theoremstyle{plain}

\begin{document}

\def\a{\alpha}
 \def\b{\beta}
 \def\e{\epsilon}
 \def\d{\delta}
  \def\D{\Delta}
 \def\c{\chi}
 \def\k{\kappa}
 \def\g{\gamma}
 \def\Ind{\mathrm{Ind}}
 \def\t{\tau}
\def\ti{\tilde}
 \def\N{\mathbb N}
 \def\Q{\mathbb Q}
 \def\Z{\mathbb Z}
 \def\C{\mathbb C}
 \def\F{\mathbb F}
 \def\ovF{\overline\F}
 \def\bfN{\mathbf N}
 \def\cG{\mathcal G}
 \def\cT{\mathcal T}
 \def\cX{\mathcal X}
 \def\cY{\mathcal Y}
 \def\cC{\mathcal C}
 \def\cD{\mathcal D}
 \def\cZ{\mathcal Z}
 \def\cO{\mathcal O}
 \def\cW{\mathcal W}
 \def\cL{\mathcal L}
 \def\bfC{\mathbf C}
 \def\bfZ{\mathbf Z}
 \def\bfO{\mathbf O}
 \def\G{\Gamma}
 \def\bO{\boldsymbol{\Omega}}
 \def\bgo{\boldsymbol{\omega}}
 \def\go{\rightarrow}
 \def\do{\downarrow}
 \def\ra{\rangle}
 \def\la{\langle}
 \def\fix{{\rm fix}}
 \def\ind{{\rm ind}}
 \def\rfix{{\rm rfix}}
 \def\diam{{\rm diam}}
 \def\uni{{\rm uni}}
 \def\diag{{\rm diag}}
 \def\Irr{{\rm Irr}}
 \def\Syl{{\rm Syl}}
 \def\Out{{\rm Out}}
 \def\Tr{{\rm Tr}}
 \def\M{{\cal M}}
 \def\cE{{\mathcal E}}
\def\td{\tilde\delta}
\def\tx{\tilde\xi}
\def\DC{D^\circ}
\def\ext{{\rm Ext}}
\def\Ker{{\rm Ker}}
\def\hom{{\rm Hom}}
\def\End{{\rm End}}
 \def\rank{{\rm rank}}
 \def\soc{{\rm soc}}
 \def\Cl{{\rm Cl}}
 \def\A{{\sf A}}
 \def\sP{{\sf P}}
 \def\sQ{{\sf Q}}
 \def\SSS{{\sf S}}
  \def\SQ{{\SSS^2}}
 \def\St{{\sf {St}}}
 \def\p{\ell}
 \def\ps{\ell^*}
 \def\SC{{\rm sc}}
 \def\supp{{\sf{supp}}}
  \def\cR{{\mathcal R}}
 \newcommand{\tw}[1]{{}^#1}

\def\Der{{\rm Der}}
 \def\Sym{{\rm Sym}}
 \def\PSL{{\rm PSL}}
 \def\SL{{\rm SL}}
 \def\Sp{{\rm Sp}}
 \def\GL{{\rm GL}}
 \def\SU{{\rm SU}}
 \def\GU{{\rm GU}}
 \def\SO{{\rm SO}}
 \def\PO{{\rm P}\Omega}
 \def\Spin{{\rm Spin}}
 \def\PSp{{\rm PSp}}
 \def\PSU{{\rm PSU}}
 \def\PGL{{\rm PGL}}
 \def\PGU{{\rm PGU}}
 \def\Iso{{\rm Iso}}
 \def\Stab{{\rm Stab}}
 \def\GO{{\rm GO}}
 \def\Ext{{\rm Ext}}
 \def\E{{\cal E}}
 \def\l{\lambda}
 \def\ve{\varepsilon}
 \def\Lie{\rm Lie}
 \def\s{\sigma}
 \def\O{\Omega}
 \def\o{\omega}
 \def\ot{\otimes}
 \def\op{\oplus}
 \def\oc{\overline{\chi}}
 \def\pf{\noindent {\bf Proof.$\;$ }}
 \def\Proof{{\it Proof. }$\;\;$}
 \def\no{\noindent}
\def\hal{\unskip\nobreak\hfil\penalty50\hskip10pt\hbox{}\nobreak
 \hfill\vrule height 5pt width 6pt depth 1pt\par\vskip 2mm}

 \renewcommand{\thefootnote}{}

\newtheorem{theorem}{Theorem}
 \newtheorem{thm}{Theorem}[section]
 \newtheorem{prop}[thm]{Proposition}
 \newtheorem{conj}[thm]{Conjecture}
 \newtheorem{question}[thm]{Question}
 \newtheorem{lem}[thm]{Lemma}
 \newtheorem{lemma}[thm]{Lemma}
 \newtheorem{defn}[thm]{Definition}
 \newtheorem{cor}[thm]{Corollary}
 \newtheorem{coroll}[theorem]{Corollary}
\newtheorem*{corB}{Corollary}
 \newtheorem{rem}[thm]{Remark}
 \newtheorem{exa}[thm]{Example}
 \newtheorem{cla}[thm]{Claim}

\numberwithin{equation}{section}
\parskip 1mm

\title[Brauer-Fowler]{Variants of some of the Brauer-Fowler Theorems}

\author[Guralnick]{Robert M. Guralnick}
\address{Department of Mathematics, University of Southern California, Los Angeles,
CA 90089-2532, USA}
\email{guralnic@usc.edu}
 
\author[Robinson]{Geoffrey R. Robinson}
\address{Department of Mathematics, King's College, 
Aberdeen,  AB24 3FX, UK}
\email{g.r.robinson@abdn.ac.uk}

 
\date{\today}

\thanks{The first author was partially supported by the NSF
grant DMS-1600056.
The paper is partially based upon work 
supported by the NSF under grant DMS-1440140 while the first author was in residence at 
the Mathematical Sciences Research Institute in Berkeley, California, during the Spring 2018
semester. It is a pleasure to thank the Institute for support, hospitality, and stimulating environment. Both authors are grateful to G. Malle for careful reading of earlier versions of this manuscript and helpful comments.}



\begin{abstract}  Brauer and Fowler noted restrictions on the structure of a finite group $G$
in terms of $|C_G(t)|$ for an involution $t \in G$.  We consider variants of these themes.   
We first note that for an arbitrary finite group $G$ of even order,
we have $$|G|< k(F)|C_{G}(t)|^{4}$$ for each involution $t \in G,$ where $F$ denotes the 
Fitting subgroup of $G$ and $k(F)$ denotes the number of conjugacy classes of $F$. In particular, 
for such a group $G$ we have $$[G:F(G)] < |C_{G}(t)|^{4}$$ for each involution $t \in G.$ This result requires the classification of the finite simple groups.

The groups ${\rm SL}(2,2^{n})$ illustrate that the above exponent $4$ cannot be replaced by any exponent less than $3.$ 
We do not know at present whether the exponent $4$ can be improved in general, though we note that the exponent $3$ suffices when $G$ is almost simple.

We are however able to prove that every finite group $G$ of even order contains an involution $u$ with  $$[G:F(G)] < |C_{G}(u)|^{3}.$$  The proof of this fact reduces to proving two residual cases: one in which $G$ is almost simple (where the classification of  the finite simple groups is needed) and one when $G$ has a Sylow $2$-subgroup of order $2$. For the latter result, the classification of finite simple groups is not needed (though the Feit-Thompson odd order theorem is).

We also prove a very general result on fixed point spaces of involutions in finite irreducible linear groups which does not make use of the classification of the finite simple groups, and some other results on the existence of non-central elements (not necessarily involutions) with large centralizers in general finite groups.

Lastly we prove (without using the classification of finite simple groups) that if $G$ is a finite group and $t \in G$ is an involution, then all prime divisors of $[G:F(G)]$ are less than or equal to $|C_{G}(t)|+1.$ 

\end{abstract}

\maketitle

\section{Introduction}   

Let $G$ be a finite group.  In \cite{GR}, the authors studied the commuting probability (i.e. the probability that two random elements commute) of $G,$ denoted here by ${\rm cp}(G).$ We have ${\rm cp}(G) = k(G)/|G|,$ 
where $k(G)$ is the number of conjugacy classes of $G$.   

Throughout this paper, an involution $t$ of a finite group $G$ will mean an element of order two in $G$.
In \cite{GR} it is noted that whenever $t$ is an involution of $G,$ we have 
$$|G| <  k(G)|C_{G}(t)|^{2},$$  (see the proof of Theorem 14 of \cite{GR}, which also implicitly yields that 
$$|I(G)| < \sqrt{k(G)|G|},$$ where $I(G)$ is the set of all involutions of $G$ ).
It is also proved in \cite[Theorem 10]{GR} that 
$$k(G) \leq \sqrt{|G|}\sqrt{k(F)},$$ where 
$F = F(G)$ is the Fitting subgroup of $G$.  Comparing these expressions easily yields $\sqrt{|G|} < \sqrt{k(F)}|C_{G}(t)|^{2}$ for every involution $t \in G$ and that, more precisely, we have $$|I(G)| < |G|^{\frac{3}{4}}k(F)^{\frac{1}{4}}.$$
Hence, in particular, we do have $$|G| < k(F)|C_{G}(t)|^{4} \leq |F(G)||C_{G}(t)|^{4}$$ for each involution $t \in G,$ so we have proved:

\begin{theorem}\label{result 1}    Let $G$ be a finite group of even order  and let $t$ be an involution of $G$.  Then $$[G:F(G)] < |C_{G}(t)|^{4}.$$
\end{theorem} 

By considering the groups $G = {\rm SL}(2,2^a)$, we see that the exponent $4$ in the Theorem cannot be improved to anything less than $3$.  
We do not know whether the result is true with the exponent $4$ replaced by $3$.  
Also, the examples $G = {\rm PSL}(2,7)$ and $G = M_{11}$ show that the bounds  $$[G:F(G)] < |C_{G}(x)|^{4}$$ and 
$$[G:F(G)] < |C_{G}(y)|^{5}$$ fail respectively for every element $x$ of order $3$ and  
every element $y$ of order $5$.  In fact, the alternating groups $A_{p+2}$ ($p$ an odd prime) show that there is no fixed integer $m$ 
such that whenever $G$ is a finite group of order divisible by a prime $p,$ then we have  
$$[G:F(G)] < |C_{G}(x)|^{m}$$ for some element $x \in G$ of order $p,$ since $A_{p+2}$ has a self-centralizing 
Sylow $p$-subgroup of order $p.$ For any fixed integer $m$, we may choose a sufficiently large prime $p$ with 
$(p+2)! > 2p^{m}.$

  We do prove:

\begin{theorem} \label{result 2}  Let $G$ be a finite group of even order.  Then there exists an involution $t \in G$ with
$$[G:F(G)] < |C_G(t)|^3.$$
\end{theorem}

A critical case that arises in the proof yields the next result which may be of independent interest. The key case is when $N = [N,\alpha],$ in which case the statement simplifies somewhat.


\begin{theorem} \label{odd case}
Suppose that $N$ is a finite group of odd order and $\alpha$ is an automorphism of $N$ of order $2$.
Then $[N:F(N)] \leq |C_{N}(\alpha)|^{3}.$ In fact, if $M = [N, \alpha]$ is nilpotent, we have 
$$[N:F(N)] \leq |C_{N}(\alpha)|,$$  while if $M$ has Fitting height $h \geq 2,$ then we have 
$$[N:F(N)] < \frac{|C_{F_{h-2}(M)}(\alpha)||C_{M}(\alpha)||C_{N}(\alpha)|}{2^{h-3}} \leq \frac{2}{3}\left( \frac{|C_{N}(\alpha)|^{3}}{2^{h-2}}\right).$$
\end{theorem}

To prove this, we require a very general result about fixed points of involutions which we feel has interest in its own right, and whose proof does not require the classification of finite simple groups:

\begin{theorem} \label{linear fix} 
Let $G$ be a finite irreducible subgroup of ${\rm GL}(V)$ containing more than one involution, where $V$ is a finite dimensional vector space over a field $\mathbb{F}.$ Then there is an involution $t \in G$ with $${\rm dim}_{\mathbb{F}}(C_{V}(t)) \geq \frac{{\rm dim}_{\mathbb{F}}(V)}{3}.$$
\end{theorem}

We also remark that Theorem \ref{result 1} (which does make implicit use of the classification of finite simple groups) gives a short proof of an extension of a famous  Theorem of Brauer and Fowler \cite{BF}
to the case of a (not necessarily simple) general finite group $G$ of even order (with a very explicit and reasonably small bound on the possibilities for $[G:F(G)]$): for a 
fixed positive integer $c$, there are only finitely many possibilities for the isomorphism type of $G/F(G)$ when 
$G$ contains an involution $t$ with $|C_{G}(t)| = c. $ We remark that the existence of \emph{some} bound on the index of the Fitting subgroup in terms of $c$ was previously known (but not made explicit) from work of Fong \cite{Fong} and Hartley and Meixner \cite{Hartley} (Fong proves some results for general primes $p$, but in the case $p=2,$ the results do not require the classification of finite simple groups).

We note that, in general, one cannot replace $F(G)$ by an Abelian normal subgroup, or even by an Abelian subgroup of maximal order. For consider the group $G=ET$ where $E$ is an extraspecial
group of odd order $p^{1+2a}$ and $T$ a group of order two acting on $E$, inverting $E/Z(E)$ elementwise (and centralizing $Z(E)$).
Then,  $|C_G(T)|=2p$ while an Abelian subgroup of $G$ of maximal order has order $p^{1+a}$ and so has index $2p^a$.  

If we allow elements other than involutions, we can prove without using the classification of finite simple groups (but using the Feit-Thompson odd order theorem and (in a degenerate case) a result of Griess \cite{Griess}):

\begin{theorem} \label{big centralizer} Let $G$ be a non-Abelian finite group. Then there is an element $x \in G \backslash Z(G)$ such that $$|G| < \frac{6|C_{G}(x)|^{3}}{5}.$$ 
\end{theorem}

Using the classification of finite simple groups, this result can be improved slightly, yielding $|G| < |C_G(x)|^3$ for some non-central $x \in G$.  We give the proof of both of these results in Section \ref{Big Centralizers}.  The groups $\SL(2,q)$ show that the exponent $3$ cannot be improved in general.

In the last section of the paper, we prove (without using the classification of the finite simple groups): 

\begin{theorem} \label{big prime} Let $G$ be a finite group of even order and let $t$ be an involution of $G$. 
Then each prime divisor of $[G:F(G)]$ is less than or equal to $|C_{G}(t)|+1.$
\end{theorem}

This bound is best possible in general- it is  
attained in the simple groups ${\rm SL}(2,p-1)$ for $p \geq 5$ a Fermat prime, for example.  
   
The paper is organized as follows.  In section \ref{aux}, we prove some auxiliary results,  including some arguments related to those of Brauer and Fowler.   Then, in section \ref{linfix}, we prove Theorem \ref{linear fix}.

In section \ref{almost simple}, we prove Theorem \ref{result 2} for almost simple groups (and stronger versions) and more generally in the case that
$F(G)=1$.  After that, we give a quick proof of Theorem \ref{big centralizer} in section \ref{Big Centralizers}.     

In section \ref{theorem2-1}, we reduce the proof of Theorem \ref{result 2} to the case where $G$ has a 
Sylow $2$-subgroup of order $2$.  In section \ref{solvcase}, we complete the proof of the remaining case of Theorem \ref{result 2} by proving Theorem \ref{odd case}. 

In section \ref{bigprimes}, we prove Theorem \ref{big prime} and another result of a similar nature. We also prove a generalization of a conjecture of Gluck \cite{Gluck} (stated as a question in \cite{CH}) in a very special case. Here, $b(G)$ denotes the largest degree of a complex irreducible character of $G$.

\begin{theorem} \label{special gluck}
Let $G$ be a finite group with one conjugacy class of involutions and with all complex irreducible representations realizable over $\mathbb{R}.$ Then $$[G:F(G)] < b(G)^{3}.$$
\end{theorem}

\section{Some auxiliary results} \label{aux} 

We record some results which are well known, and one new result which we need later. We also fix some notation to be used later, and recall some standard notation and terminology.

As usual, we define the Fitting series of a finite group $X$ by: $F_{1}(X) = F(X)$ (the Fitting subgroup of $X$, that is to say, the unique largest nilpotent normal subgroup of $X$) and $$F_{i+1}(X)/F_{i}(X) = F(X/F_{i}(X))$$ for $i \geq 1.$ We adopt the convention that $F_{0}(X) = 1.$ When $X$ is solvable, the Fitting height of $X$ is defined to be the smallest non-negative integer $h$ such that $X = F_{h}(X).$

For any finite group $X$, the generalized Fitting subgroup of $X$ is denoted by $F^{\ast}(X),$ and is the central product $F(X)E(X),$ where $E(X)$ is the central product of all components (that is, quasisimple subnormal subgroups) of $X$. We say that $X $ is almost simple if $F^{\ast}(X)$ is a non-Abelian simple group.

We also recall that the upper central series of the finite group $X$ is defined by: $$Z_{1}(X) = Z(X)$$ and $$Z_{i+1}(X)/Z_{i}(X) = Z(X/Z_{i}(X))$$ 
for $i \geq 1.$

The first result of the section is an easy observation which ultimately reduces to the case where $X$ is an elementary Abelian $p$-group with $p$ odd. 

\begin{lemma} \label{bw}   {\rm[Brauer-Wielandt]} Let $A$ be a Klein $4$-group with involutions $a,b,c$ which acts as a group of automorphisms on a finite group $X$ of odd order: Then 
$$|X| = \frac{|C_{X}(a)||C_{X}(b)||C_{X}(c)|}{|C_{X}(A)|^{2}}.$$
\end{lemma} 

One standard result on coprime groups of automorphisms which we frequently use is:

\begin{lemma} \label{cc} Let $X$ be a finite group, and $A$ be a subgroup of ${\rm Aut}(X)$ with 
$${\rm gcd}(|X|,|A|) = 1.$$ Then whenever $Y$ is an $A$-invariant normal subgroup of $X,$ we have
$$C_{X/Y}(A) \cong C_{X}(A)/C_{Y}(A).$$
\end{lemma}

The next two results are known to specialists, (see for example Theorem E of \cite{MW}), but since they are quite important for our result, we give short proofs. 


\begin{lemma} \label{nil-reg}   Let $X$ be a nilpotent group of odd order acting faithfully on a nilpotent group $V$  with $|V|$ also odd and with 
$\gcd(|X|,|V|)= 1$.   Then $X$ has at least two regular orbits on $V$. In particular, 
$$|X| \leq \frac{|V|-1}{2}.$$
\end{lemma}  

\begin{proof}   Note that it suffices to prove the existence of single regular orbit for $X$ since if $Xv$ is a regular orbit, then 
so is $Xv^{-1} \ne Xv$.   Since $X$ and $V$ have coprime orders, we can pass to $V/\Phi(V)$ and so assume that $V$ is a direct product
of elementary Abelian $p$-groups.   We now use additive notation for $V$.   If $V$ is not irreducible as an $X$-module, write $V=V_1 \oplus V_2$ with each $V_i$ invariant under $X$. 
Then $v_1 + v_2$ with $v_i \in V_i$ is contained in a regular orbit if each $v_i$ is (for $X/X_i$ where $X_i$ is the kernel of $X$ acting on $V_i$).   

So we may assume that $V$ is an irreducible $\mathbb{Z} /p\mathbb{Z} X$-module for some odd prime $p$ which does not divide $|X|$.   Let $F = \mathrm{End}_X(V)$.   Then
$F$ is a field of size $q=p^a$ for some $a$.  We view $V$ as an absolutely irreducible $FX$-module (the underlying set $V$ has not changed).   Thus, we can write 
$V=V_1 \otimes \ldots \otimes V_m$ where $V_i$ is an absolutely irreducible $kS_i$-module with $S_i$ a Sylow $r_i$-subgroup of $X$ (with the $r_i$ the primes dividing $|X|$).
Thus, it suffices to prove the result when $X$ is an $r$-group for some odd prime $r \ne p$.   

We now proceed by induction on $\dim_F V$.  If $\dim V=1$, this is clear (using that $q > 2|X|$).
If ${\rm dim}V >1,$ then since $V$ is absolutely irreducible, it is induced from an irreducible $kM$-module $W$ where $M$ is a maximal normal subgroup of $X$ of index $r$.
So we may identify $V$ with $W \oplus \ldots \oplus W$ where $M$ acts on each of the $r$ copies of $W$ and if $x \in X \setminus{M}$, $x$ permutes the $r$ copies of $W$ cyclically.
We may also assume that $x$ has order $r$ (otherwise any stabilizer will intersect $M$ nontrivially).

We may enlarge $M$ and assume that $M= M_1 \times \ldots \times M_r$ where $M_i$ is the projection of $M$ acting on the $i$th copy of $W$ (so $X \cong M_1 \wr \mathbb{Z}/r\mathbb{Z}$).
By induction, there exists $v \in W$ so that $M_1v$ is a regular orbit and (changing notation if necessary) we may assume that $v$  generates a regular orbit for $M_i$ as well.
Thus, $M(v, \ldots, v)$ is a regular orbit for $M$. 
Consider the set of $2^r$ vectors of the form $(\varepsilon_{1} v, \varepsilon_{2}v, \ldots, \varepsilon_{r}v )$ in $V,$ where each $\varepsilon_{i} \in \{1,-1\}$.  Note that each of these vectors generates a regular orbit for $M$ 
(for if $1 \ne y \in M$ fixes one of these vectors, then $y^2$ fixes $(v,\ldots,v)$ whence $y^2=y=1$).   The same argument shows that all of
these orbits are distinct for $M$.  Clearly, $x$ permutes this set of size $2^r$ with precisely $2$ fixed points.   Thus, $x$ fixes none of the other
$2^r-2$ orbits of $M$ and so we have produced at least $\frac{2^r-2}{r}$ regular orbits for $X$.
\end{proof}

\begin{lemma} \label{spare reg} Let $G$ be a finite group of odd order. Then 
$$[F_{2}(G):F(G)] \leq \frac{|F(G)|-1}{2}.$$
\end{lemma}

\begin{proof}
The statement is unaffected if we replace $G$ by $G/\Phi(G),$ or if we replace $G$ by $F_{2}(G),$ 
so we may and do suppose that $\Phi(G) = 1$ and that $G$ has Fitting height two.

Now $F=F(G)$ is a direct product of minimal normal subgroups of $G$, so may be viewed as a completely reducible 
module for the nilpotent group $G/F.$ 

\medskip
Let $$F = V_{1} \times V_{2} \times \ldots \times V_{r}$$ where each $V_{i}$ is a minimal normal subgroup of $G$
and thus $V_{i}$ is a $p_{i}$-group for some prime $p_{i}.$                                        
Then the nilpotent group $G/F$ embeds (faithfully) in the nilpotent group 
$$G_{1} \times G_{2} \times \ldots \times G_{r},$$ where $$G_{i} = G/C_{G}(V_{i})$$ for each $i.$ 
We note that the nilpotent group $G_{i},$ which acts faithfully and  irreducibly on $V_{i},$ is a nilpotent $p_{i}^{\prime}$-group.
Now we may apply Lemma \ref{nil-reg} to conclude that $G_{i}$ has at least two regular orbits on $V_{i}$ for each $i.$ Now we have 
$$[G:F(G)] \leq |G_{1}| \times |G_{2}| \times \ldots \times |G_{r}| 
\leq \frac{|V_{1}|-1}{2} \times \ldots \times \frac{|V_{r}|-1}{2} \leq \frac{|F(G)|-1}{2}.$$ 

\end{proof}   
 
We remark that the upper bound given in this result is attained infinitely often, for example by every Frobenius group of order $\frac{q(q-1)}{2}$ with $q \equiv 3$ (mod $4$) a prime power.

The next result will be useful in the proof of Theorem \ref{linear fix}.

\begin{lemma} \label{Hall}
Let $A$ be a finite group of automorphisms of a finite solvable group $G$ with ${\rm gcd}(|G|,|A|) = 1.$ Let $\pi$ be a set of prime divisors of $|G|$. Then any $A$-invariant Hall $\pi$-subgroup of $G$ contains a Hall $\pi$-subgroup of $C_{G}(A).$ 
\end{lemma}

\begin{proof} 
$G$ has an $A$-invariant Hall $\pi$-subgroup, $K$ say, and by a Theorem of Glauberman, any other $A$-invariant Hall $\pi$-subgroup of $G$ is conjugate to $K$ via an element of $C_{G}(A).$ Also, every $A$-invariant $\pi$-subgroup of $G$ is contained in an $A$-invariant Hall $\pi$-subgroup of $G$. A Hall $\pi$-subgroup $L$ of $C_{G}(A)$ is certainly $A$-invariant, so is contained in some $A$-invariant Hall $\pi$-subgroup of $G,$ which has the form $K^{x}$ for some $x \in C_{G}(A).$ Then $L^{x^{-1}}$ is a Hall $\pi$-subgroup of $C_{G}(A)$ contained in $K.$
\end{proof}

We will also need a slight strengthening of the fact that (as noted during the course of the proof of Lemma 3 of \cite{GR})  when $G$ is a finite solvable group of Fitting height at most two, then we have $k(G) \leq |F(G)|.$ 

\begin{lemma} \label{Fitting height two Hall} Let $G$ be a finite solvable group and $\pi$ be a set of primes.
Suppose that $O_{\pi^{\prime}}(G) = 1$ and that $G$ has a normal Hall $\pi$-subgroup $H$ of Fitting height at most two. Then $k(G) \leq |F(G)| = |F(H)|.$ 
\end{lemma}

\begin{proof}
Notice that $F(G)$ is a $\pi$-group. We may suppose that $\Phi(G) = 1,$ so that $F = F(G)$ is a direct product of various minimal normal subgroups of $G,$ so is completely reducible as $G/F$-module. Let $V,$ a $p$-group for some prime $p \in \pi,$ be one of these.
We note that $G/C_{G}(V)$ is a $p^{\prime}$-group. Firstly, $H/F$ is a nilpotent group acting completely reducibly on $V$ (since $H/F \lhd G/F $), so that $H/C_{H}(V)$ is a $p^{\prime}$-group. Also, $G/H$ is a $\pi^{\prime}$-group, so that $G/C_{G}(V)$ is a $p^{\prime}$-group, as claimed.

Let $M$ be a maximal subgroup of $G$ which does not contain $V$ (such exists as $\Phi(G) = 1).$  Note that 
$M \cap V = 1$ by minimality of $V$. Then $$C_{G}(V) = V \times C_{M}(V)$$ and $C_{M}(V) \lhd G.$ Then $G/C_{M}(V)$ is the semidirect product of 
$V$ with $$M/C_{M}(V) \cong G/C_{G}(V),$$  so we have $$k(G/C_{M}(V)) \leq |V|$$ by the solution of the $k(GV)$-problem (for solvable groups). If $V = F,$ we are done.

Otherwise, we know that $U = F(C_{M}(V))$ is a $G$-invariant complement to $V$ in $F,$ since $C_{M}(V) \lhd G.$
We may suppose by induction that $k(C_{M}(V))  \leq |U|$ since $C_{M}(V)$ satisfies the hypotheses of the 
lemma (as $C_{M}(V) \lhd G$). Since we have $$k(G) \leq k(G/C_{M}(V))k(C_{M}(V)),$$ we do have $k(G) \leq |U||V| = |F|,$ as required.
\end{proof}

We return to the main theme of the paper to close this section. We recall that  the Frobenius-Schur indicator of a complex irreducible character $\chi$, denoted by $\nu(\chi)$, takes value $1$ if $\chi$ may be afforded by a real representation, $-1$ if $\chi$ is real-valued, but may not be afforded  by any real representation, and $0$ if $\chi$ is not real-valued. The following well-known formula allows us to count involutions using characters.

\begin{lemma} \label{invcount} The number of involutions in the finite group $G$ of even order is
$$\sum_{1 \neq \chi \in {\rm Irr}(G)} \nu(\chi) \chi(1).$$
\end{lemma}

\medskip
Now let $G$ be a finite group with more than one conjugacy class of involutions. Let $t$ and $u$ be non-conjugate involutions of $G$. Then (as Brauer-Fowler noted in \cite{BF}), 
for each $g \in G$,  there is some involution $z \in C_{G}(t)$ such that $u^{g} \in C_{G}(z)$ (just take $z$ to be an involution central in the dihedral group 
$\langle t,u^{g} \rangle$,  (such an involution exists because $\langle tu^{g} \rangle$ is a cyclic normal subgroup of even order of $ \langle t,u^{g} \rangle ))$.

Let $c$ be the maximum order of the centralizer of an involution in $G.$ Then it was further noted in \cite{BF}
that 
$$[G:C_{G}(u)] < c^{2},$$ so that certainly $|G| < c^{3}$.   For note that each conjugate of $u$ is in the union 
$$\bigcup_{z \in C_{G}(t)^{\#}:z^{2} = 1} C_{G}(z)^{\#}.$$ 

We record this result of Brauer and Fowler as:

\begin{lemma} \label{more than 1}   If $G$ is a finite group with at least two conjugacy classes of involutions, then
$|G| < |C_G(u)|^3$ for some involution $u \in G$.
\end{lemma}

\section{A general bound on fixed point spaces of involutions in linear actions} \label{linfix}

In this section, we prove Theorem \ref{linear fix} without recourse to the Classification of Finite Simple Groups. We suppose that $G$ is a finite irreducible subgroup of ${\rm GL}(V),$ where $V$ is a finite dimensional vector space over a field of characteristic $q$ and that $G$ contains more than one involution. We need to prove that $G$ contains an involution $t$ with $${\rm dim}(C_{V}(t)) \geq \frac{{\rm dim}(V)}{3}.$$ This is clear if $q=2,$ so suppose otherwise.

We remark that if an irreducible subgroup $X$ of ${\rm GL}(V)$ (for $V$ as above) contains a unique involution, then that involution is central in $X.$ 

If $G$ contains a Klein $4$-group $A,$ then the result follows from Lemma \ref{bw} if $q \neq 0$ or by consideration of  the multiplicities of the linear constituents of ${\rm Res}^{G}_{A}(\chi),$ where $\chi$ is the character afforded by $V$, when $q = 0$.

Hence we may suppose that $G$ contains no Klein $4$-subgroup, so that $G$ has cyclic or generalized quaternion
Sylow $2$-subgroups.  Then by the Brauer-Suzuki Theorem, we have $G = O_{2^{\prime}}(G)C_{G}(t)$  for any involution $t \in G.$ Set $N = O_{2^{\prime}}(G),$ and let $T = \langle t \rangle$ for $t \in G$ an involution. Set $M = [N,T]T \lhd G$ and note that we still have $G = MC_{G}(t).$

We claim that if suffices to assume that $G = M.$ For let $W$ be an irreducible summand of ${\rm Res}^{G}_{M}(V)$ (note that all these are $G$-conjugate by Clifford's Theorem, so since $G = C_{G}(t)M$ we see that the dimension of the $t$-fixed point space on $W$ is independent of the choice of the irreducible summand $W$. In particular, $t$ does not act as a scalar on $W$). Note also that ${\rm Res}^{G}_{M}(V)$ is completely reducible by Clifford's theorem since $M \lhd G.$  Let $X = M/C_{M}(W),$ which is isomorphic to a subgroup of ${\rm GL}(W)$ with a non-central involution $tC_{M}(W)).$

If $X$ is a proper section of $G$ then we may assume that the result is true for $X$ by induction. In that case, 
we have $${\rm dim}(C_{W}(t)) \geq \frac{{\rm dim}(W)}{3},$$ and then $${\rm dim}(C_{V}(t)) \geq \frac{{\rm dim}(V)}{3},$$ as desired.

Hence we may suppose that $G$ has a Sylow $2$-subgroup $T$ of order $2$ (generated by $t,$ say), and a normal $2$-complement $N$ with $N = [N,t].$ In particular, $G$ is now solvable. We may also suppose that there is no proper subgroup $Y$ (normal or not) of $G$ with $t \in Y$ and $G = YC_{G}(t).$ For if $G = YC_{G}(t),$ then $$[N,t] = [G,t] = [Y,t] \leq Y \cap N$$ and we have $G = [Y \cap N,t]C_{G}(t), $ in which case $G = LC_{G}(t)$ with 
$L = [Y \cap N,T]T \lhd G,$ which is a proper normal subgroup of $G$. As above, we may take an irreducible summand $W$ of ${\rm Res}^{G}_{L}(V)$ and work with $X = L/C_{L}(W),$ which is a proper section of $G$, noting as before that $t$ does not act as a scalar on $W,$ and deduce that $${\rm dim}(C_{W}(t)) \geq \frac{{\rm dim}(W)}{3},$$ and $${\rm dim}(C_{V}(t)) \geq \frac{{\rm dim}(V)}{3}.$$  Hence from now on, we suppose that there is no such proper subgroup $Y.$

Now we may reduce to the case that $V$ is a $\mathbb{C}G$-module affording the irreducible character $\chi$. This is routine if $q=0.$

In the case that $q \neq 0,$ (over some finite extension field, which may be chosen as ${\rm End}_{G}(V)$ if desired) we may write $V$ as a sum of Galois conjugate absolutely irreducible $G$-modules. Let $U$ be one of these, which may be chosen so that $|U| = |V|$. Then, since $G$ is solvable, $U$ lifts to a faithful irreducible $\mathbb{C}G$-module affording character $\chi.$ Furthermore, the dimension of the $t$-fixed point space on $U$ is $\frac{\chi(1) + \chi(t)}{2}$, which is independent of the lift chosen.

Then $$\frac{{\rm dim}(C_{U}(t))}{{\rm dim}(U)} = \frac{{\rm dim}(C_{V}(t))}{{\rm dim}(V)} = \frac{\chi(1) + \chi(t)}{2\chi(1)},$$ so the ratio of the dimension of the $t$-fixed point space to the dimension of the whole space is unchanged by passage to the complex representation affording $\chi.$

From now on, then, we suppose that $V$ is a $\mathbb{C}G$-module,  and we work with the character $\chi$ it affords. We still have $G = NT,$ where $[N,T] = N,$ and we recall that we are supposing that there is no proper subgroup $Y$ of $G$ with $t \in Y$ and $G = YC_{G}(t).$  To prove the desired result, it suffices to prove that $$\frac{\chi(1)+\chi(t)}{2} \geq \frac{\chi(1)}{3}.$$

Now we claim that we may suppose that $\chi$ is primitive. If not, then since $G$ is solvable, $\chi$ may be induced from an irreducible character $\mu$ of some maximal (but not necessarily normal) subgroup $M$ whose index is a power of a prime $r.$ If $M$ contains no conjugate of $t,$ then we have $\chi(t) = 0,$ so that $$\frac{\chi(1)+\chi(t)}{2} = \frac{\chi(1)}{2},$$ and we are done. Hence we may suppose that $t \in M$ and that $r$ is odd.

Whenever $t^{g} \in M$ for $g \in G,$ we have $t^{g} = t^{m}$ for some $m \in M,$ (since $T$ is a Sylow $2$-subgroup of $M$), so that $g \in C_{G}(t)M$ (which need not be a subgroup). It follows that $$\chi(t) = [C_{G}(t):C_{M}(t)]\mu(t),$$ while we know that $\chi(1) = [G:M]\mu(1).$

Now $M \cap N$ contains a Hall $r^{\prime}$-subgroup of $N,$ so by Lemma \ref{Hall}, we know that $C_{M \cap N}(t)$ contains a Hall $r^{\prime}$-subgroup of $C_{N}(t).$ Hence both $[G:M]$ and $[C_{G}(t):C_{M}(t)]$ are powers of $r,$  since $r$ is odd. It is elementary that $$[C_{G}(t):C_{M}(t)] \leq [G:M],$$ and in our case, the inequality is strict, since we know that $G \neq MC_{G}(t).$ However, both $[C_{G}(t):C_{M}(t)]$ and $[G:M]$ are powers of $r,$ so we now have $$r[C_{G}(t):C_{M}(t)] \leq [G:M].$$

In particular, we now have 
$$|\chi(t)| = [C_{G}(t):C_{M}(t)]|\mu(t)| \leq \frac{[G:M]\mu(1)}{r} = \frac{\chi(1)}{r}.$$
Hence we have $$\frac{\chi(1)+ \chi(t)}{2} \geq \frac{r-1}{2r}\chi(1) \geq \frac{\chi(1)}{3},$$ (since $r$ is odd) and we are done.

Hence we suppose from now on that $\chi$ is primitive. Then every Abelian normal subgroup of $G$ is cyclic and central. Now $T \not \leq F(G)$ since $t$ does not act as scalars on $V$ and $T \in {\rm Syl}_{2}(G).$ 
Hence $F = F(G) = F(N)$ has odd order. Since all characteristic Abelian subgroups of $F$ are cyclic, we see that there is a prime $p$ such that  $$P = \Omega_{1}(O_{p}(G))$$ is a non-Abelian extraspecial $p$-group with $Z(P) \leq [P,t] \neq 1.$ Note that since $\chi$ is primitive and faithful, but $G$ is non-Abelian, $G$ is not nilpotent.

By (Clifford's Theorem and) primitivity, all irreducible constituents of ${\rm Res}^{G}_{P}(\chi)$ are equal, say to $\theta.$ Note that $\theta$ extends to $PT$ (in just two possible ways in fact) by Glauberman correspondence, though we can see it directly here. Let $\psi$ be an irreducible constituent of ${\rm Res}^{G}_{PT}(\chi).$ Then $\psi$ restricts to $P$ as an integer multiple of $\theta$, but the multiple cannot be two (or more) using the character inner product. Hence there is at least one extension, $\psi, $ which does not vanish identically on the $2$-section of $t$ in $PT$ (which is just the coset $Pt$ in this situation). If we multiply $\psi$ with the linear character of order $2$ of $PT,$ we obtain another extension of $\theta$. One and only one of the two extensions has trivial determinant on restriction to $T$.

By Glauberman correspondence, there is an irreducible character $\eta$ of $C_{P}(t)$ 
such that $\psi(t) = \pm \eta(1).$ Now $\psi(t) \neq \pm \psi(1),$ since $1 \neq Z(P) \leq [P,t]$
and $\theta$ is a faithful character of $P.$ Since $\eta(1)$ and $\psi(1)$ are both powers of $p,$ we see that 
$\eta(1) \leq \frac{\psi(1)}{p}.$ Hence we have 
$$\frac{\psi(1)+\psi(t)}{2} \geq \frac{(p-1)\psi(1)}{2p}.$$

Since $\psi$ was an arbitrary irreducible constituent of ${\rm Res}^{G}_{PT}(\chi),$ we deduce that 
$$\frac{\chi(1)+\chi(t)}{2} \geq \frac{(p-1)\chi(1)}{2p} \geq \frac{\chi(1)}{3},$$ since $p \geq 3$.
The proof of Theorem \ref{linear fix} is now complete.

A somewhat related (though different) result may be found in Liebeck-Shalev \cite{Liebshal}, though our result applies (in particular) to solvable groups.

As a final note on this result, we remark that it is not true in general that under the same hypotheses one can find an involution
whose $-1$ eigenspace has dimension at least $(\dim V)/3$.  Consider $G= {\rm U}(3,3),$ which has a complex irreducible character $\chi$ of degree $7$ with $\chi(t)  = 3$ for every involution $t,$ whence the fixed space of $t$ has dimension $5$ and the $-1$ eigenspace of $t$ has dimension $2$ in the associated representation. Among finite simple groups, this is in fact the only counterexample (using the fact that any other non-Abelian finite simple group can be generated by $3$ involutions \cite{MSW}).

 \section{Almost Simple Groups} \label{almost simple} 

In this section, we prove some results about centralizers in quasisimple and almost simple groups.  

\begin{thm}  \label{as}   Let $G$ be a finite almost simple group.  Then we have $|C_G(t)|^3 > |G|$ for every involution $t \in G.$ 

\end{thm}

We do not require the full strength of this result in order to prove Theorem \ref{result 2}.  All we need is that there is some such involution and as noted in
Lemma \ref{more than 1}, it suffices to assume that $G$ has a unique class of elements of order $2$.   The almost simple groups with
a single class of involutions are quite restricted and it is quite easy to prove the result in that case.   However, we record the stronger result since it would
be required to prove a stronger result for the general finite group of even order.

Our proof uses the classification of finite simple groups.   As noted before, it follows by \cite[Theorem 14]{GR}, it would suffice to know
that $|C_G(t)| \ge k(G)$.  This is almost always true (but it does barely fail for $G = {\rm SL}(2,2^a)$ when $k(G)  = 2^a+1 > 2^a= |C_G(t)|$ for $t$
an involution).  

If $F^*(G)$ is a sporadic group or $A_n$ with $n \le 12$,  then this follows from \cite{Atlas} or \cite{GL}.   Suppose that $G=S_n$ 
with $n > 12$.   If $t$ is an involution moving precisely $2d$ points, then $|C_G(t)| = (2^d)(d!)(n-2d)!$ and the result
is straightforward (and similarly for $A_n$).

So we may assume that $S:=F^*(G)$ is a simple group of Lie type of rank $r$ over the field of $q$ elements. 
    Then the involutions in $G$ are either inner-diagonal, field automorphisms, graph automorphisms
or field-graph automorphisms. Their centralizers are described in \cite{GLS}.   In the last three cases, the conjugacy classes of
such are given and in all cases, the centralizer has order roughly $|G|^{1/2}$ and the result follows.  Indeed, we also note
that $k(G)$ in these cases is only a bit larger than $q^r$ \cite{FG}.  
  
If $S$ is an exceptional group, then the inner diagonal involutions are given in \cite{Lu} and the result follows by observation
(for the case of even characteristic, see also \cite{LS}).    

It remains to consider the case that $S$ is a classical group (i.e. linear, unitary, symplectic or orthogonal) and the involution is
inner diagonal.   If $S$ has odd characteristic, there are two possibilities : the first is that an involution will induce an orthogonal decomposition
of the space and the centralizer will be a product of two large classical subgroups and the result follows easily.   The other possibility
is that the element lifts to an element of larger order in the simply connected group and the centralizer will have order approximately 
$|G|^{1/2}$.  
If $S$ has characteristic $2$, then the centralizers are described in \cite{LS} and again it is straightforward to check the result.

We give a proof with more detail of the case we really need.

\begin{thm}  Let $S$ be a finite non-Abelian simple group and $G$ be a finite group with $F^*(G)=S$ such that  $G$ has a unique conjugacy class of involutions. Then $[G:S]$ is odd,  $S$ has a unique conjugacy class of involutions, and $|C_G(t)|^3 > |G|$ for any involution $t \in G$.
\end{thm}

\begin{proof}  By \cite[12.1]{FGS},  any finite non-Abelian simple group $S$ contains an involution whose $S$-conjugacy class is invariant under $\mathrm{Aut}(S)$ and so, in our case, under $G$.
Clearly $[G:S]$ is odd and since $G$ has a unique class of involutions and the $S$-class is $G$-invariant, $S$ has a unique class of involutions.
In particular, we see that $G=SC_G(t)$ and so it suffices to assume that $G=S$ to prove the inequality.  

The list of simple groups containing a unique conjugacy class of involutions is quite limited.  One can check the sporadic groups as above.
If $S= A_n, n \ge 5$, then the only possibilities are $n \le 7$.  Groups of Lie type can have rank at most $2$ and for those we just need
to check the inequality for any involution (since either this is the unique class or there is more than one).
\end{proof}

We note that this gives:

\begin{thm} \label{big involution centralizer}   Let $G$ be a finite group with $F(G)=1$.  Then there exists an involution $t \in G$ with
$|C_G(t)|^3 > |G|$.
\end{thm}

\begin{proof}   Let $E=E(G)$ be the direct product of the minimal normal subgroups of $G$.  Since
$F(G)=1$, $E$ is a direct product of $r$ non-Abelian simple groups.  If $r > 1$, then $G$ has more than 
one conjugacy class of involutions (choose one involution contained in a simple factor and another
which projects nontrivially in at least two factors) and the result follows by Lemma \ref{more than 1}.   
So we may assume that $G$ is almost simple and the previous result applies.
\end{proof}  

We need the following variation of the previous result (which has an easier proof but still depends on the classification).

\begin{lemma} \label{quasisimple}  Let $G$ be a finite group with $E:=F^*(G)$ quasisimple.  Then there exists  
$x \in G \setminus{Z(G)}$ such that $|C_G(x)|^3 > |G|.$  
\end{lemma}

\begin{proof}
  We just produce a non-central element of the
quasisimple group  $E$ with a very large centralizer.   Let $S=E/Z(E)$.   If $S$ is alternating choose $x$ to be
a $3$-cycle.   If $n > 7$, then $x$ has order prime to the Schur multiplier and the centralizer has order bigger than
$|G|^{1/2}$.  If $n \le 7$, one just checks directly.  Similarly, one checks the result holds for $S$ a sporadic simple group \cite{Atlas}.

Finally, if $S$ is a simple group of Lie type in characteristic $p$, we usually take $x$ to be a long root element.  In a few cases of twisted groups, we can take
any non-central element of $E$ which is central in a Sylow $p$-subgroup $P$ of $E$.  
We choose the  conjugacy class of $x$ to be  invariant under field automorphisms and graph automorphisms (this is possible aside from the case
of $B_2$ or $F_4$  in characteristic $2$ and $G_2$ in characteristic $3$).   In almost all cases, $p$ does not divide  $|Z(E)|$.   One checks directly that in 
 all cases we have  $|P|^3 > |E|$ and the invariance of the class under field automorphisms implies that $|C_G(x)|^3 > |G|$.
 \end{proof}

\section{Big Centralizers} \label{Big Centralizers}

We will prove Theorem \ref{big centralizer}.  We start with some easy observations.  So let $G$ be a finite non-Abelian group with Fitting subgroup $F=F(G)$.
We first show that in many cases we have $|G| < |C_G(x)|^2$ for some  non-central $x \in G$.  If $A \lhd G$ is Abelian, but non-central, then we have 
$[G:C_{G}(x)] < |A|$ for any $x \in A^{\#}$,  so we certainly have $|G| < |C_{G}(x)|^{2}$. If $Z_{2}(G) > Z(G),$ we may take $A = \langle x,Z(G) \rangle$ for $x \in Z_{2}(G) \backslash Z(G)$.
If $F = F(G)$ is Abelian, but not central, we may take $A = F$ and obtain the same conclusion.   More generally as long as $F \ne Z(G)$, we can choose
$x \in F$ not central in $G$ with $[x,F] \le Z(G)$.  This implies that $[G:C_G(x)] < |F|$.   Since $[F:C_F(x)] <  |Z(G)|$, we obtain 
$$
 |G| < |F||C_G(x)| < |C_G(x)||C_F(x)||Z(G)| \le |C_G(x)|^2|Z(G)|.
 $$
 
 We record this:
 
 \begin{lemma}  Let $G$ be a finite group with $F(G) \ne Z(G)$.  Then there exists $x \in G \setminus{Z(G)}$ with $|G| < |C_G(x)|^2|Z(G)|.$
 \end{lemma}
 
Hence we now consider a finite group $G$ with $F = F(G) = Z(G).$

Since $F=Z(G)$ and $G$ is non-Abelian, it follows that 
  $G$ has a component $L$,  which may be chosen so that $|L|$ is minimal among all components. Set $E = E(G)$.  Suppose that $E$ has $r > 1$
  components. Then $[E:C_E(x)] < |L|$ for $x \in L \backslash Z(L)$ and $[G:C_G(x)] < r|L|$, 
  so that certainly 
  $$
  |G| < r |L||C_G(x)| < |C_G(x)|^2.
  $$ 

We are reduced to considering the case that $F^{\ast}(G) = Z(G)L$ for a single component $L$. We first prove the slightly weaker result without the full classification
(but using the Feit-Thompson Theorem, and a Theorem of Griess \cite{Griess} in the extreme case that all involutions of $G$ are central). If all involutions of $G$ are central, then Griess's Theorem allows  
us to conclude that $L$ is either some ${\rm SL}(2,q)$ for 
$q >3$ odd, or else $L/Z(L) \cong A_{7}$ and $|Z(L)| \in \{2,6\}$.  In these cases, it can be seen by inspection that there is some $x \in G \backslash Z(G)$ with $|C_{G}(x)|^{3} > |G|.$ 
Hence we suppose that $G$ contains a non-central involution $t$.  

In that case, (using \cite{GR} again), we have $$|G| < k(G)|C_{G}(t)|^{2}.$$ Now choose $x \in G \backslash Z(G)$ with $|C_{G}(x)|$ as large as possible.  Then certainly $$|C_{G}(x)| \geq |C_{G}(t)|.$$ Furthermore, we have $$k(G) < |C_{G}(x)| + |Z(G)|$$ from the proof of Lemma 2, part vii) of \cite{GR}.
Since $G$ is certainly not solvable, we know that there is an element $y \in G$ such that $yZ(G)\in G/Z(G)$ has prime order $p \geq 5.$ Then $$|C_{G}(x)| \geq |C_{G}(y)| \geq 5|Z(G)|.$$
Hence we now have $$|G| < \frac{6|C_{G}(x)|^{3}}{5}.$$

In order to prove the better inequality stated at the end of Theorem 5, the proof above shows that it suffices to consider the almost quasisimple case.  
Now apply Lemma \ref{quasisimple}.   

 \section{Theorem \ref{result 2}: Some Reductions}  \label{theorem2-1}

We now consider the question of lowering the exponent $4$ with the relaxed condition that we wish to prove that 
whenever $G$ is a finite group of even order, there is \emph{some} involution $t \in G$ with 
$$[G:F(G)] < |C_{G}(t)|^{3}.$$ 
We have seen by Lemma \ref{more than 1}
 that this condition is satisfied if $G$ has more than one conjugacy class of involutions, 
 so we suppose that $G$ has only one conjugacy class of involutions. We also assume that $G$ has minimal (even) order subject to having 
 $$[G:F(G)] \geq |C_{G}(t)|^{3}$$ 
 for \emph{every} involution $t \in G.$ Then $G$ is certainly not nilpotent.
 
 We have proved the result if $F= F(G)=1$ in Section \ref{almost simple}, so suppose that $F >1.$  
If $O_{2}(G) \neq 1,$ then all involutions of $G$ lie in 
$$Z = \Omega_{1}(Z(O_{2}(G)))$$ since $G$ has a single conjugacy class of involutions.
Hence for each involution $t \in G$, we have 
$$[G:C_{G}(t)] \leq |Z|-1,$$  so we certainly have $$|G| < |C_{G}(t)|^{2},$$ contrary to the choice of $G.$ 
Hence $O_{2}(G) = 1.$ 
 
So $1 \ne F$ has odd order. Let $N = O_{2^{\prime}}(G).$ Then $G/N$ still has a unique conjugacy class of involutions. If $O_{2}(G/N) \neq 1,$ let $Z \lhd G$ be the full preimage in $G$ 
of $\Omega_{1}(Z(O_{2}(G/N))).$ Applying the above argument in $G/N,$ we have 
$$[G:N] < [C_{G}(t):C_{N}(t)]^{2}$$ 
for each involution $t \in G.$

If now $G$ contains a Klein $4$-group $A,$ then all involutions of $A$ are conjugate to $t$ above, and Lemma \ref{bw} gives $|N| \leq |C_{N}(t)|^{3}.$  
Then $|G| < |C_{G}(t)|^{2}|C_{N}(t)|,$ contrary to the choice of $G$. Thus if $Z>N,$ then $G$ has cyclic or generalized quaternion Sylow $2$-subgroups. But 
in that case, (by the Brauer-Suzuki Theorem), $G = NC_{G}(t)$ for any involution $t \in Z$.  Now let $L = N\langle t \rangle \lhd G$.
  If $L$ is proper, then we have 
  $$[G:C_{G}(t)] = [L:C_{L}(t)] < |F(L)||C_{L}(t)|^{2}$$ by the minimality of $G$,
  which certainly yields $$[G:F(G)] < |C_{G}(t)|^{3}$$ since $F(L) \leq F(G).$ 
  Hence $L = G$ in this case, and $G$ has a normal $2$-complement and a Sylow $2$-subgroup of order $2$.

Suppose then that $O_{2}(G/N) = 1.$ Then since we certainly have $O_{2^{\prime}}(G/N) = 1,$ we see that $F(G/N) = 1.$ Now $G/N$ still has a single conjugacy class of involutions, and since $N \neq 1,$ the minimal choice of $G$ gives (since $F(G/N) = 1$) that 
$$[G:N] < [C_{G}(t):C_{N}(t)]^{3}$$ for any involution $t \in G.$ Also, $G$ does not have cyclic or 
generalized quaternion Sylow $2$-subgroups (again by the Brauer-Suzuki theorem). 
Hence $G$ contains a Klein $4$-subgroup, say $A$, and we have $$|N| \leq |C_{N}(t)|^{3}$$ for the involution $t$ above. Hence we obtain 
$$|G| < |C_{G}(t)|^{3}$$ in this case, contrary to the choice of $G.$ 

In summary, we have shown that it suffices to prove the theorem in the case that  $G$ has a Sylow $2$-subgroup of order $2$.

\section{Theorem \ref{result 2}: The proof in the solvable case} \label{solvcase}

We turn our attention to proving the desired result in the remaining case, so we suppose that our minimal counterexample $G$ has a Sylow $2$-subgroup of order $2,$ and we 
set $$N = O_{2^{\prime}}(G).$$ 

In this case, we will prove a sharper version of the result: we will prove  Theorem \ref{odd case} which asserts
in particular that whenever $N$ is a non-trivial (solvable) group of odd order and $t$ is an automorphism of $N$ of order $2,$ then we 
have $$[N:F(N)] \leq  |C_{N}(t)|^{3}.$$ Set $M = [N,t] \lhd N.$ Then $N = MC_{N}(t),$ so that 
$$[N:F(M)] = [C_{N}(t):C_{M}(t)][M:F(M)].$$ If $M$ is nilpotent, then we have $$[N:F(N)] \leq |C_{N}(t)|$$
and we are done.

Hence we may suppose that $M = [N,t]$ is not nilpotent, and that $N = M,$ since $[N:M] = [C_{N}(t):C_{M}(t)].$

Now  $N = [N,t]$ has Fitting height $h \geq 2.$ We will now prove the second inequality of Theorem \ref{odd case} by induction on $|N|.$ We recall that we are adopting the convention that $F_{0}(X) = 1.$

Note that we may suppose that $Z(N) = 1, $ since replacing $N$ by 
${\bar N} = N/Z(N)$ does not change the index of the Fitting subgroup, and we have 
$$|C_{{\bar N}}({\bar t})| \leq |C_{N}(t)|$$ (likewise with $C_{F_{h-2}({\bar N})}({\bar t})$).

In particular, this implies that no minimal normal subgroup $V$ of $G$ is centralized, or inverted elementwise,
by $t$ (for otherwise, $V$ is centralized by $[N,t] = N$).

We note that $N/F_{h-2}(N)$ is not nilpotent, so is certainly non-Abelian, and $t$ must therefore have non-trivial fixed points in its action on $N/F_{h-2}(N).$ In particular, $$|C_{N}(t)| \geq 3|C_{F_{h-2}(N)}(t)|$$
since $N$ has odd order.

Suppose now that $h=2,$ so that $N/F(N)$ is nilpotent, but $N$ is not. We have $$\Phi(G) = \Phi(N),$$ and we may suppose that $\Phi(N) = 1.$ Then $F = F(N) = F(G)$ is a completely reducible module for $H = G/F.$ Set $L = N/F.$ 
By Theorem \ref{Fitting height two Hall}, applied with $\pi$ the set of odd prime divisors of $|G|$ and with the Hall subgroup $N,$ we know that $$k(G) \leq |F(N)| = |F(G)|.$$

By Theorem 14 of \cite{GR}, we know that $$|G|< k(G)|C_{G}(t)|^{2},$$ so that we obtain
$$|N| < 2|F(N)||C_{N}(t)|^{2}.$$ Here, $2^{h-3} = 2^{-1},$ so this is in accord with the statement 
of the Theorem. Furthermore, since $N$ is not nilpotent, we have $C_{N}(t) > 1,$ so we do have 
$$[N:F(N)] <  \frac{2|C_{N}(t)|^{3}}{3},$$ since $$|C_{N}(t)| \geq 3.$$ This completes the proof of the Theorem in the case that $[N,t]$ has Fitting height $2.$

Now we continue to assume that $N = [N,t],$ but now we suppose that $N$ has Fitting height $h >2.$ 

By induction, we may suppose that $$[L:F(L)] < \frac{|C_{F_{h-3}(L)}(t)||C_{L}(t)|^{2}}{2^{h-4}},$$
since $L = [L,t]$ has Fitting height $h-1\geq 2.$ 

Now we have $$[N:F(N)] = [N:F_{2}(N)][F_{2}(N):F(N)] \leq \frac{|F|-1}{2}[L:F(L)],$$ using Lemma \ref{spare reg}.
By applying Theorem \ref{linear fix} to each minimal normal subgroup $V$ of $G$ contained in $N$, recalling that $t$ does not act as a scalar on any such $V$, we also have $$|F| < |C_{F}(t)|^{3}.$$ Now we have  $$[N:F(N)] < \frac{|C_{F}(t)|^{3}}{2}\frac{|C_{F_{h-3}(L)}(t)||C_{L}(t)|^{2}}{2^{h-4}}.$$ 

Since  $$F_{h-2}(N)/F = F_{h-3}(L)$$ and $L = N/F,$ we may write this as
$$[N:F(N)] < \frac{|C_{F_{h-2}(N)}(t)||C_{N}(t)|^{2}}{2^{h-3}}, $$ as required.

Since $N/F_{h-2}(N)$ is not nilpotent and $|N|$ is odd, we have $$3|C_{F_{h-2}(N)}(t)| \leq |C_{N}(t)|,$$
so we have  $$[N:F(N)] < \frac{2}{3}\left( \frac{|C_{N}(t)|^{3}}{2^{h-2}}\right).$$

This completes the proof of Theorem \ref{odd case} and also the proof of 
Theorem \ref{result 2}.

\section{Commuting probabilities, character degrees, involution centralizers and large prime divisors} \label{bigprimes}

For a finite group $G$, we let $b(G)$ denote the maximal degree of a complex irreducible character of $G$.
D. Gluck has conjectured in \cite{Gluck} that when $G$ is solvable, we should have 
$$[G:F(G)] \leq b(G)^{2}.$$ 
This question remains open at present. A. Moret\'o and T. Wolf \cite{MW} proved that we do have 
$$[G:F(G)] \leq b(G)^{3}$$  when $G$ is solvable, and it has been conjectured (or rather, proposed as a question in \cite{GR}) that this inequality should hold for general finite $G$ (the examples ${\SL}(2,2^{n})$ indicate that the exponent $3$ can not be lowered in general).  We refer to this inequality from now on as the general form of Gluck's conjecture. It is proved in \cite{CH} (using \cite{GR})  that we do have $[G:F(G)] \leq b(G)^{4}$ for general finite groups $G$.

Theorem \ref{result 2} yields a positive answer to the general form of Gluck's conjecture in the special case stated in Theorem \ref{special gluck} :

\begin{proof} Let $G$ be a finite group whose irreducible characters all have Frobenius-Schur indicator $1$, and suppose   further that $G$ has one conjugacy class of involutions with representative $t.$ Then by Lemma \ref{invcount} we see that 
$$[G:C_{G}(t)] =  \sum_{ 1 \neq \chi \in {\rm Irr}(G)}\chi(1) \geq  \frac{|G|-1}{b(G)}.$$ 
Since $b(G)$ divides $|G|$, we deduce that $$|C_{G}(t)| \leq b(G).$$  By Theorem \ref{result 2} (since $G$ has a single conjugacy class of involutions),  we have $$[G:F(G)] < |C_{G}(t)|^{3} \leq b(G)^{3}.$$  
\end{proof}

This suggests the possibility that there might be some relationship between the general form of Gluck's conjecture and sizes of centralizers of involutions. 

As remarked in the introduction, we noted in \cite{GR} that $$|G| < k(G)|C_{G}(t)|^{2}$$ whenever $t$ is an involution of the finite group $G$. This may be written as $${\rm cp}(G) > \frac{1}{|C_{G}(t)|^{2}}$$ for each involution $t \in G,$ where ${\rm cp}(G)$ denotes the commuting probability of $G$. 

On the other hand, we proved (by relatively elementary means) in \cite{GR} that we have 
${\rm cp}(G) \leq \frac{1}{p}$ whenever $p$ is a prime such that $G$ does not have a normal 
Sylow $p$-subgroup. Hence whenever $G$ is a finite group containing an involution $t,$
we have $$p < |C_{G}(t)|^{2}$$ for each prime divisor $p$ of $[G:F(G)].$  

When $G$ is simple, this also follows from Brauer-Fowler \cite{BF}, since they proved that $|G|$ divides  $\left( |C_{G}(t)|^{2}\right)!$ in that case.

In fact, we can do somewhat better than this. Using theorems of Brauer \cite{Brauer}, Feit \cite{Feit} and Feit and Thompson \cite{FT} on complex linear groups, but avoiding use of the classification of finite simple groups, we now prove Theorem \ref{big prime}:

\begin{proof} Suppose that the Theorem is false and let $G$ be a minimal counterexample. Let $P$ be a Sylow $p$-subgroup of $G$ where $$p \geq |C_{G}(t)|+2$$ is a prime, and $t$ is an involution of $G$. Note that since $p$ is odd and $|C_{G}(t)|$ is even,  we in fact have $$p \geq |C_{G}(t)|+3.$$

We will derive a contradiction, thus proving that $P \lhd G$. Then we may also conclude that $t$ inverts $P$ elementwise, since $p$ does not divide $|C_{G}(t)|,$ so that $P$ is Abelian. It may be helpful to note first that if $N$ is a normal subgroup of $G$ which contains neither $t$ nor $P$, then $G/N$ inherits the hypotheses since 
$$|C_{G/N}(tN)| \leq |C_{G}(t)| \leq p-2$$ and $G/N$  has order divisible by $p.$ 

\medskip
Suppose that $t \in H \lhd G$ with $H$ proper. Then every $G$-conjugate of $t$ lies in $H$, so that 
$$[G:C_{G}(t)] < |H|$$ and $$[G:H] < |C_{G}(t)| < p.$$ Thus $P \leq H$, and then $P \lhd H$ by minimality of $G$. Then also $P \lhd G,$ a contradiction. Hence no proper normal subgroup of $G$ contains $t,$  so that $$G = \langle t^{g}: g \in G \rangle. $$ 

\medskip
We note that $Z(G) = 1$ since $t \not \in Z(G).$ For if $Z(G) \neq 1,$ then $G/Z(G)$ has a normal Sylow $p$-subgroup by minimality of $G$, and then $G$ does too, contrary to assumption. Similarly, we have $O_{p}(G) = 1,$ for otherwise $P/O_{p}(G) \lhd G/O_{p}(G)$ by minimality of $G$ and then $P \lhd G,$ a contradiction.

\medskip
We note also that $G$ has a unique minimal normal subgroup, $K$ say, since otherwise $G$ embeds in $$G/K_{1} \times G/K_{2}$$ for $K_{1},K_{2}$ distinct minimal normal subgroups of $G,$ and both factors have a normal (though possibly trivial) Sylow $p$-subgroup by the minimality of $G$, so $G$ has a normal Sylow $p$-subgroup, contrary to hypothesis. 

\medskip
If $K$ is a $q$-group for some prime $q \neq p,$ then $$PK/K \lhd G/K = {\bar G}$$ by minimality of $G.$ Then ${\bar t}$ inverts ${\bar P} \cong P$ elementwise. Now $G = \langle t \rangle PK$ by minimality of $G$ (otherwise $[P,K] = 1$ and $P \lhd G$). Also, $G$ is now solvable and $F(G) = K$ (for $F(G)$ is a $p^{\prime}$-group containing $K$ but not containing $t$). Note also that $t$ now normalizes every subgroup of $PK$ containing $K.$ Now it follows that $|P| = p$ by the minimality of $G$ (otherwise there is a subgroup $S$ of order $p$ of $P$ such that $\langle t \rangle SK$ is a subgroup, and then $[S,K] = 1,$ contradicting $K = F(G)).$ Since ${\bar t}$ inverts ${\bar P}$ elementwise, $t$ inverts some $p$-singular element of $G$. Hence $G = KD$ where $D$ is a dihedral subgroup of $G$ of order $2p$ containing $t,$ and we may suppose that $P = O_{p}(D).$ 

\medskip
Note that (by minimality of $K$), $P$ acts semi-regularly by conjugation on non-identity elements of $K,$ so that $|K| \equiv 1$ (mod $p$).

\medskip
On consideration of the structure of irreducible $FD$-modules, where $F = \mathbb{Z}/q\mathbb{Z}$, we see that $$|C_{K}(t)| = |K|^{\frac{1}{2}}$$ 
(even if $q = 2$). Hence $$p > |C_G(t)| \geq  2|C_{K}(t)| = 2|K|^{\frac{1}{2}}.$$ However, 
$p$ divides  $$(|K|^{\frac{1}{2}}-1)(|K|^{\frac{1}{2}}+1)$$ and $$|K|^{\frac{1}{2}}= |C_{K}(t)|$$ is an integer, so that $p \leq |K|^{\frac{1}{2}} + 1,$ a contradiction.

\medskip
Hence $K$ is now a direct product of non-Abelian simple groups permuted transitively under conjugation by $G$
and $F(G) = 1.$ Note that we now have $$K = E(G) = F^{\ast}(G)$$ and $C_{G}(K) = 1.$ We claim that $K$ is simple, so that $G$ is almost simple. Suppose otherwise. 

\medskip
If $p$ divides $|K|,$ then $t$ normalizes no component $L$ of $K$, otherwise $\langle t\rangle L$ is proper and hence $L$ has a normal Sylow $p$-subgroup by the minimality of $G$, contrary to the simplicity of $L$. Now there is a component $L$ of $G$ with $L^{t} \neq L.$ But then $t$ centralizes the diagonal subgroup of $LL^{t},$ and $p$ divides $|C_{G}(t)|,$ a contradiction. Hence $K$ is a $p^{\prime}$-group when $K$ is not simple.

\medskip
Now as before we obtain $G = KD$ with $D$ dihedral of order $2p$ and $t \in D.$ Again, we may suppose that $P = O_{p}(D).$ Since we are assuming that $K$ is not simple, $K$ is a direct product of $2,p$ or $2p$ components.
If there are just two components, say $L,L^{t},$ then both are normalized by $P.$ However $t$ centralizes the diagonal subgroup of $LL^{t}$ so that $|L| \leq |C_{G}(t)| < p$ and $[P,L] = 1$ (since all $P$-conjugacy classes of $L$ have size one), likewise for $L^{t}$. Then $[P,K]=1,$ a contradiction. If the number of components of $K$ is divisible by $p,$ then we see that $$p > 2|C_{K}(t)| \geq 2^{\frac{p+1}{2}},$$ a contradiction (we even have $$|C_{K}(t)| \geq |L|^{\frac{p-1}{2}}$$ for $L$ a component of $K$).

\medskip
Hence $K$ is simple, as claimed, and $G$ is almost simple. Furthermore, if $|K|$ is not divisible by $p$, we find as above that $G = KD$ with $D$ dihedral of order $2p,$ and we may suppose that $D = P\langle t \rangle.$ 
If $p$ divides $|K|$, then we either have $G = K$ or $G = \langle t \rangle K.$ 

\medskip
We use character theory to deal with these three possibilities.

\medskip
Suppose first that $G = K$. Then by Lemma \ref{invcount}, we have $$[G:C_{G}(t)] \leq \sum_{ \{1 \neq \chi \in {\rm Irr}(G): \nu(\chi) = 1\} } \chi(1).$$
Hence we certainly have $$[G:C_{G}(t)] < \frac{|G|}{d}$$ where $d$ is the minimal degree of a non-trivial irreducible character $\chi$ of $G$ such that $\nu(\chi) = 1.$ This yields 
$$\chi(1) = d < |C_{G}(t)| \leq p-3,$$ so that $d \leq p-4.$ Note that $\chi$ is faithful as $G$ is simple.

\medskip
By a Theorem of Feit, \cite{Feit} we know that $|P|=p,$ since $\chi$ is faithful.

\medskip
We may embed $G$ in ${\rm GL}(d,R)$ (still affording character $\chi$) where $(K,R,F)$ is a $p$-modular system.
Let $V$ be the underlying $RG$-module. Since $d < p,$ Green correspondence tells us that 
$${\rm Res}^{G}_{N_{G}(P)}(V)$$ is indecomposable (because $${\rm Res}^{G}_{N_{G}(P)}(V)$$ is projective-free).

\medskip
It follows that all indecomposable summands of ${\rm Res}^{G}_{P}(V)$ have equal rank, which is the rank of a source of $V$. If there is only one indecomposable summand, necessarily of rank $d$, then $O_{p^{\prime}}(C_{G}(P))$ acts as scalars on $V,$ so that $C_{G}(P) = P.$

\medskip
If there are two or more summands, then a generator of $P$ has minimum polynomial of degree less than $\frac{p-1}{2}$ in any reduction (mod $p$) of $G$. By \cite{GRR}, we have $O_{p^{\prime}}(C_{G}(P)) = 1$
and $C_{G}(P) = P.$

\medskip
In any case, then, we have $C_{G}(P) = P.$ Let $e$ be the inertial index $$[N_{G}(P):C_{G}(P)]$$ which must be less than $p-1$ since $$1 < \chi(1) < p-1.$$
Since $C_{G}(P) = P,$ we see that $\chi(1) \in \{e,p-e\}.$ If $e = \frac{p-1}{2},$ we have 
$$G \cong {\rm PSL}(2,p)$$ by a Theorem of Brauer \cite{Brauer} (using the facts that $Z(G) = 1$ and $\chi$ is faithful). Hence we suppose this is not the case (an involution centralizer  in ${\rm PSL}(2,p)$ has order $p \pm 1).$ Hence $e < \frac{p-1}{2},$ so that we cannot  have $\chi(1) = e$  by Feit-Thompson \cite{FT} again. Hence $\chi(1) = p-e.$ 

\medskip
Since $\chi$ is real-valued and exceptional, we must have that $e$ is even, so now $\chi(1)$ is odd.

\medskip
Let $\theta$ be the sum of the exceptional characters in the principal $p$-block. Then we have 
$\theta(x) = 1$ for every generator $x$ of $P$, (note that $\theta(1) = (p-e)\left( \frac{p-1}{e}\right)
\equiv 1$ (mod $p$)). Since  $\chi(1)$ is odd, we know that $\chi(t) \neq 0.$ Furthermore all exceptional characters take the same value at $t$.

\medskip
We claim that some conjugate of $t$ lies in $N_{G}(P).$ If this is not the case, then we 
have $$\sum_{ \mu \in {\rm Irr}(B)} \frac{\mu(t)^{2} \mu(x)}{\mu(1)} = 0$$ by the standard Burnside-Frobenius character count  of the number of times $x^{-1}$ is expressible as a product of two conjugates of $t$, where $B$ denotes the principal $p$-block of $G$ (which we now know to be the only $p$-block of positive defect).

\medskip
The contribution to the sum from the trivial character and the exceptional characters is 
$$1 + \frac{\chi(t)^{2}}{p-e}  >1 $$ since $\chi(t) \neq 0.$ However, the non-trivial non-exceptional characters all have degree congruent to $\pm 1$ (mod $p$), so all of these have degree at least $p-1$ (and all take value $\pm 1$ at $x$).

\medskip
Certainly if we sum $\alpha(t)^{2}$ over the non-trivial non-exceptional irreducible characters in the principal $p$-block, the result is at most $$|C_{G}(t)| -1 \leq p-4.$$

\medskip
Hence $$\sum_{ \alpha} \frac{\alpha(t)^{2} |\alpha(x)|}{\alpha (1)} \leq \frac{p-4}{p-1} < 1$$ where $\alpha$ runs through the non-trivial non-exceptional characters in the principal $p$-block.

\medskip
Now we have  $$\sum_{ \mu \in {\rm Irr}(B)} \frac{\mu(t)^{2} \mu(x)}{\mu(1)} \geq  1 - \frac{p-4}{p-1} >0,$$ a contradiction.

\medskip
We now know that some conjugate of $t$ lies in $N_{G}(P),$ which is a Frobenius group of order $pe$ with a cyclic Hall $p^{\prime}$-subgroup of order $e$. Hence $|C_{G}(t)|$ is divisible by $e$.

\medskip
Since $e$ also divides $(p-1)$ and $$|C_{G}(t)| \leq p-3 < p-1,$$ we obtain 
$$|C_{G}(t)| \leq (p-1-e).$$ Now, however, $$p-e = \chi(1) <  |C_{G}(t)| \leq p-1-e,$$ a contradiction.
Hence $G$ is not simple.

\medskip
Suppose now that $$G = \langle t \rangle K \neq K$$ with $K$ simple of order divisible by $p.$

\medskip
Now $G$ has exactly two linear characters (both realizable over $\mathbb{R}$), and by Lemma \ref{invcount} we have 
$$[G:C_{G}(t)] \leq 1 + \sum_{\{ \chi \in {\rm Irr}(G): \nu(\chi) = 1  \neq \chi(1)\}}\chi(1).$$
 Hence we certainly have $$[G:C_{G}(t)] \leq \frac{|G|}{d}$$ where $d$ is the minimal degree of a non-linear irreducible character $\chi$ of $G$ such that $\nu(\chi) = 1.$ This yields $$\chi(1) = d \leq  |C_{G}(t)| \leq p-3.$$ Note that $\chi$ is faithful.

\medskip
As before, by a Theorem of Feit, \cite{Feit} (and using simplicity of $K$),  we know that $|P| = p$.

\medskip
Arguing as in the simple case, we have $C_{G}(P) = P$ and $\chi(1) \in \{e,p-e\}$ with $e < \frac{p-1}{2},$ and we can't have $\chi(1) = e$ by \cite{FT} again. Again, $e$ is even and $\chi(1) = p-e$ is odd.

\medskip
Furthermore, ${\rm Res}^{G}_{K}(\chi)$ is irreducible, for otherwise, it decomposes as a sum of two irreducible characters of equal degree, contrary to the fact that $\chi(1)$ is odd.

\medskip
Now we must also have $e = [N_{K}(P):C_{K}(P)]$ since ${\rm Res}^{G}_{K}(\chi)$ must be an exceptional character in the principal $p$-block of $K$. This yields a contradiction because $G = KN_{G}(P)$ and 
we  have $$2 = [G:K] = [N_{G}(P):N_{K}(P)] = \frac{[N_{G}(P):C_{G}(P)]}{[N_{K}(P):C_{K}(P)]}= 1,$$ a contradiction (we have used the fact that $C_{G}(P) = C_{K}(P) = P).$

\medskip
Now we have reduced to the case that $G = KD$ where $D$ is dihedral of order $2p = [G:K]$ (and $K$ is non-Abelian simple of order not divisible by $p$).

\medskip
Then since $G/K \cong D$, we know that $G$ has exactly $\frac{p+3}{2}$ irreducible characters with $K$ in their kernel (all real-valued) and the sum of their degrees is $p+1.$ All other irreducible characters of $G$ are faithful. Then (again by Lemma \ref{invcount}) we have $$[G:C_{G}(t)] \leq p + \frac{|G|- (p+1)}{d} < p +\frac{|G|}{d},$$ where $d$ is the minimal degree of a faithful real-valued irreducible character of $G$.

\medskip
If $d < p,$ then $P$ centralizes a $P$-invariant Sylow $q$-subgroup of $K$ for every odd prime $q$
by the Hall-Higman-Shult theorem, so that $K = RC_{K}(P)$ for a $P$-invariant Sylow $2$-subgroup $R$ of $K$.
Then $$[K,P] \leq R$$ is a (normal) $2$-subgroup of $K$, which forces $[K,P] = 1$  by the simplicity of $K.$
But in this case, $F^{\ast}(G) = K,$ a contradiction. Hence $d \geq p.$ If $d \geq 2p-6,$ then we obtain
$$d \geq 2|C_{G}(t)|$$ and $$\frac{2|G|}{d} < p + \frac{|G|}{d}.$$ Then  $$p > \frac{|G|}{d},$$ and $$d^{2} >|G|,$$ a contradiction, since $d$ is the degree of an irreducible character of $G$. Hence $$p \leq d < 2p-6.$$

\medskip
Now we have $$\frac{|G|}{d-3} \leq [G:C_{G}(t)] \leq p + \frac{|G|- (p+1)}{d}.$$
Hence $$\frac{3|G|}{d(d-3)} \leq \frac{dp-p-1}{d}$$ so $$|G| < \frac{(d-3)(d-1)p}{3} < \frac{4p^{3}}{3}.$$
Thus $$|K| < p^{2}.$$ 

\medskip
However, $K$ has a $P$-invariant Sylow $q$-subgroup $Q$ for each of its prime divisors $q.$ If $|Q| < p,$ then 
$[P,Q] = 1.$ Since $|K| < p^{2},$ there is at most one prime $q$ such that $|Q| > p.$ 
If there is no such $q$, then we have $[K,P]= 1,$ a contradiction. If there is only one such $q$ with corresponding $P$-invariant Sylow $q$-subgroup $Q$, then we have $$K = QC_{K}(P),$$ so that $$[K,P] \leq Q$$ is a normal $q$-subgroup of $K$. Since $K$ is simple, we have the contradiction $[K,P] = 1.$ 

The proof of Theorem 6 is now complete.
\end{proof}

To conclude this section, we mention another result of a similar nature with elementary proof. We prove that when $G$ is a finite group of even order and $N$ is a non-trivial nilpotent subgroup of $G,$ then either 
there is an element $x \in N^{\#}$ with  relatively large centralizer, or else $|N|$ is not too much larger than the minimum value of $|C_{G}(t)|$, as $t$ ranges through involutions of $G$.

\begin{theorem} \label{bigorsmall}
Let $G$ be a finite group of even order and let $t$ be an involution of $G$. Let $N$ be a non-trivial nilpotent subgroup of $G$. Suppose that $|C_{G}(x)|^{3} < |G|$ for all $x \in N^{\#}.$ 
Then :

\medskip
\noindent i) If $N$ is Abelian, we have $|N| < \sqrt{2}|C_{G}(t)|.$

\medskip
\noindent ii) If $N$ is non-Abelian, we have  $$|N| < \sqrt{\frac{2|Z(N)|-1}{|Z(N)|-1}}|C_{G}(t)|.$$ 
\end{theorem}

\begin{proof}  The argument is analogous to one used in \cite{GR} with a Sylow $p$-subgroup in the role of $N,$ but here we may take advantage of the results of \cite{GRR3}. Choose an element $x \in N^{\#}$ so that $|C_{G}(x)|$ is as large as possible. Note that $|C_{G}(x)| \geq |N|$.

\medskip
\noindent i) If $N$ is Abelian, then by \cite{GRR3}, we have $$\sum_{n \in N^{\#}}|\chi(n)|^{2} \geq |N|-1$$ for any irreducible character $\chi$ of $G$ which does not vanish identically on $N^{\#}.$ It follows easily from the orthogonality relations that the number of irreducible characters of $G$ which do not vanish identically on $N^{\#}$ is at most $|C_{G}(x)|.$

\medskip
On the other hand, any irreducible character $\chi$ of $G$ which vanishes identically on $N^{\#}$ has degree divisible by $|N|$, so there are fewer than $\frac{|G|}{|N|^{2}}$ such characters.

\medskip
Since we are assuming that $|G| \geq |C_{G}(x)|^{3},$ we have 
$$\frac{1}{|C_{G}(t)|^{2}} < {\rm cp}(G) \leq  \frac{1}{|C_{G}(x)|^{2}} + \frac{1}{|N|^{2}} \leq \frac{2}{|N|^{2}},$$ and i) follows.

\medskip
\noindent ii) This is similar to the proof of i), except that (using the results of \cite{GRR3}), the number of irreducible characters of $G$ which do not vanish identically on $N^{\#}$ is at most 
$$\frac{|Z(N)||C_{G}(x)|}{|Z(N)|-1},$$ since whenever $\chi$ is an irreducible character of $G$ which does not vanish identically on $N^{\#},$  we have $$\sum_{n \in N^{\#}}|\chi(n)|^{2} \geq |N|-\mu(1)^{2},$$ where $\mu$ is an irreducible character of maximal degree of $N,$ and we always have $$\mu(1)^{2} \leq [N:Z(N)].$$ 

\end{proof}

We note that when $$G = {\rm SL}(2,2^{a}),$$ there is an Abelian subgroup $N$ of $G$ of order $2^{a}-1$
such that $$|C_{G}(x)|^{3} = |N|^{3} < |N|(|N|+1)(|N|+2) = |G|$$ for each element $x \in N^{\#}$ and we have $$|N| = |C_{G}(t)|-1$$ for each involution $t \in G.$

This demonstrates that the bound given in the theorem cannot 
be improved in general to  $$|N| < \alpha |C_{G}(t)|$$ for any fixed constant $\alpha < 1.$

\end{document}